\documentclass[12pt]{amsart}
\usepackage[margin=1in]{geometry}

\usepackage{amsthm,amsmath,amssymb,color,setspace}
\usepackage{mathtools,leftidx,tensor}
\usepackage[hidelinks]{hyperref}
\usepackage{graphicx}
\usepackage{xurl}

\usepackage[foot]{amsaddr}

\theoremstyle{definition}
\newtheorem{defi}{Definition}[section]
\newtheorem{ex}[defi]{Example}
\theoremstyle{plain}
\newtheorem{thm}[defi]{Theorem}

\newcommand{\lr}[1]{\lbrace #1 \rbrace}

\makeatletter
\def\imod#1{\allowbreak\mkern10mu({\operator@font mod}\,\,#1)} 
\def\@setcopyright{}                                           
\def\serieslogo@{}
\makeatother

\begin{document}
\singlespacing

\author[J.M.P.~Balmaceda]{Jose Maria P. ~Balmaceda}
\address[J.M.P.~Balmaceda]{Institute of Mathematics, University of the Philippines Diliman, 1101 Quezon City, Philippines}
\email{jpbalmaceda@up.edu.ph}

\author[D.V.A.~Briones]{Dom Vito A. ~Briones}
\address[D.V.A.~Briones, Corresponding author]{Institute of Mathematics, University of the Philippines Diliman, 1101 Quezon City, Philippines}
\email{dabriones@up.edu.ph}

\title{A survey on association schemes on triples}

\begin{abstract} 
Association schemes on triples (ASTs) are ternary analogues of classical association schemes, whose relations and adjacency algebras are ternary instead of binary. We provide a survey of the current progress in the study of ASTs, highlighting open questions, suggesting research directions, and producing some related results. We review properties of the ternary adjacency algebras of ASTs, ASTs whose relations are invariant under some group action, and ASTs obtained from 2-designs and two-graphs. We also provide a notion of fusion and fission ASTs, using the AST obtained from the affine special linear group $ASL(2,q)$ as an example.
\end{abstract}



\date{\today}

\maketitle
\singlespacing

\section{Introduction}
Classical association schemes originated from Bose and Shimamoto as partitions of a Cartesian product $\Omega\times\Omega$ with certain symmetry properties \cite{Bose1952}. These may be regarded as colorings of the edges of complete graphs satisfying desirable regularity conditions. A special case is the family of two-class association schemes which consists of colorings with two colors that yield the family of strongly regular graphs. The symmetry conditions that classical association schemes satisfy are sufficiently flexible as to accommodate various mathematical structures, yet adequately rigid as to endow classical association schemes with many desirable algebraic and combinatorial properties \cite{bannai_algebraic_1984}. For instance, the algebras generated by the adjacency matrices of classical association schemes are semisimple and, when commutative, satisfy duality properties that allow for the computation of possible parameter
values of various families of graphs.


In \cite{mesner_association_1990}, Mesner and Bhattacharya defined a ternary analogue for classical association schemes called association schemes on triples (ASTs). Instead of partitions of the Cartesian product $\Omega \times \Omega$, an AST is a partition of the Cartesian triple product $\Omega \times \Omega \times \Omega$ satisfying analogous regularity conditions. In particular, the adjacency hypermatrices of ASTs form a ternary algebra under an extension of the usual binary matrix product to a ternary cubic hypermatrix product.

In this survey, we share current progress in the study of ASTs, highlighting open questions, suggesting research directions, and producing some related results. In particular, we consider properties of the ternary adjacency algebras of ASTs, ASTs whose relations are invariant under some group action, such as symmetric ASTs, ASTs from two-transitive groups, and circulant ASTs, and the ASTs obtained from 2-designs and two-graphs. Lastly, we provide a notion of fusion and fission ASTs, using the AST obtained from the affine special linear group $ASL(2,q)$ as a primary example. The parameters of this AST are also obtained, thereby extending the work done in \cite{balmafamily}. The paper is structured as follows. We define ASTs in Section \ref{section_ast} and consider its algebraic structure as a ternary algebra in Section \ref{sect_algs}. Through Section \ref{section_grp}, we examine the relationships between group actions and ASTs. Proceeding, we explore the relationships between ASTs and 2-designs and two-graphs in Section \ref{section_combi}. Finally, we define fusion and fission ASTs in Section \ref{section_sub}, providing some examples and focusing particularly upon the AST obtained from $ASL(2,q)$.





\section{Assocation schemes on triples}\label{section_ast}
This section is based mostly on \cite{mesner_association_1990} and \cite{Zealand2021}. We define an association scheme on triples (AST) as a partition of a triple cartesian product satisfying certain symmetry conditions and then view ASTs as ternary algebras through their adjacency hypermatrices.

\begin{defi}
Let $\Omega$ be a finite set with at least 3 elements. An association scheme on triples (AST) on $\Omega$ is a partition $X=\lr{R_i}_{i=0}^m$ of $\Omega \times \Omega \times \Omega$ with $m\geq 4$ such that the following hold.

\begin{enumerate}
    \item  For each $i\in \lr{0,\ldots,m}$, there exists an integer $n_i^{(3)}$ such that for each pair of distinct $x,y\in \Omega$, the number of $z\in \Omega$ with $(x,y,z)\in R_i$ is $n_i^{(3)}$.
    \item (Principal Regularity Condition.) For any $i,j,k,l \in \lr{0,\ldots,m}$, there exists a constant $p_{ijk}^l$ such that for any $(x,y,z)\in R_l$, the number of $w$ such that $(w,y,z)\in R_i$, $(x,w,z)\in R_j$, and $(x,y,w)\in R_k$ is $p_{ijk}^l$.
    \item For any $i\in \lr{0,\ldots, m}$ and any $\sigma \in S_3$, there exists a $j\in \lr{0,\ldots,m}$ such that \[R_j= \lr{(x_{\sigma(1)},x_{\sigma(2)},x_{\sigma(3)}):(x_1,x_2,x_3)\in R_i}.\]
    \item The first four relations are $R_0=\lr{(x,x,x): x\in \Omega}$, $R_1=\lr{(x,y,y):x,y\in \Omega,x\neq y}$, $R_2=\lr{(y,x,y):x,y\in \Omega, x\neq y}$, and $R_3=\lr{(y,y,x):x,y\in \Omega, x\neq y}$.
\end{enumerate}
\end{defi}

The following example is the AST on $\Omega=\lr{1,2,3}$ with five relations.

\begin{ex}\label{example_ASTrel} Let $\Omega=\lr{1,2,3}$ and $X=\lr{R_i}_{i=0}^4$ be the partition of $\Omega \times \Omega \times \Omega$ given by the following ternary relations.
\begin{align*}
   { R_0}&=\lr{(1,1,1),(2,2,2),(3,3,3)},\\
    {R_1} &= \lr{(1,2,2),(1,3,3),(2,1,1),(2,3,3),(3,1,1),(3,2,2)},\\
    {R_2} &= \lr{(2,1,2),(3,1,3),(1,2,1),(3,2,3),(1,3,1),(2,3,2)},\\
    {R_3} &= \lr{(2,2,1),(3,3,1),(1,1,2),(3,3,2),(1,1,3),(2,2,3)},\\
    R_4 &= \lr{(1,2,3),(1,3,2),(2,1,3),(2,3,1),(3,1,2),(3,2,1)}.
        \end{align*}
Then $X$ is an AST on $\Omega$. Since $3$ is the only $w\in \Omega$ such that $(1,2,w)\in R_4$, we have $n_4^{(3)}=1$. Further, $(1,2,3)\in R_4$ and there is no $w$ such that $(w,2,3)$, $(1,w,3)$, and $(1,2,w)$ are all in $R_4$. Thus, we have $p_{444}^4=0$.
\end{ex}

The integer $n_i^{(3)}$ is the third valency of $R_i$, analogous to the valency of a classical association scheme. Conditions 1 and 3 imply for each $i$ the existence of constants $n_i^{(1)}=\vert \lr{z\in \Omega : (z,x,y)\in R_i} \vert $ and $n_i^{(2)}=\vert \lr{z\in \Omega : (x,z,y)\in R_i} \vert $ independent of any pair of distinct $x,y\in \Omega$. Accordingly, $n_i^{(1)}$ is called the first valency of $R_i$ and $n_i^{(2)}$ is called the second valency of $R_i$. The relations $R_0,R_1,R_2$ and $R_3$ are called the trivial relations and the other relations are the nontrivial relations. Further, the numbers $p_{ijk}^l$ are called the intersection numbers.

An AST can be viewed in terms of hypermatrices that generate a ternary algebra whose structure constants are the $p_{ijk}^l$, mirroring the situation between classical association schemes and their adjacency algebras. 

Let $X=\lr{R_i}_{i=0}^m$ be an AST on a set $\Omega$ of size $\nu$. We associate with each $R_i\in X$ the $\nu \times \nu \times \nu$ hypermatrix $A_i$ whose entries are indexed by $\Omega$. This $A_i$ is given by \[(A_i)_{xyz}=\begin{cases}
1, &\text{if } (x,y,z)\in R_i, \\ 0, &\text{otherwise}.
\end{cases}\] The $\mathbb{C}$-vector space generated by the adjacency hypermatrices forms a ternary algebra under the ternary operation $ABC\mapsto D$, where $D$ is the $\nu \times \nu \times \nu$ hypermatrix given by \[D_{xyz}=\sum_{w\in \Omega} A_{wyz} B_{xwz} C_{xyw}.\]

The structure constants of this ternary algebra are given by the intersection numbers $p_{ijk}^l$ of $X$ \cite{mesner_association_1990}; that is, $A_i A_j A_k = \sum_{l=0}^m p_{ijk}^l A_l$. For instance, if we consider the AST in Example \ref{example_ASTrel}, then we obtain $A_4 A_4 A_4 = \sum_{l\in \Omega} p_{444}^l A_l = 0$.





\section{Algebraic structure of ASTs}\label{sect_algs}

Currently, little is known about the structure of the ternary algebra obtained from the adjacency matrices of an AST. However, the situation is partially simplified by considering the subalgebra generated by the adjacency hypermatrices of the nontrivial relations. Indeed, Corollary 2.8 of \cite{mesner_association_1990} says that the adjacency matrices $\lr{A_i}_{i=4}^m$ of the nontrivial relations generate a subalgebra of the ternary algebra generated by all the adjacency matrices $\lr{A_i}_{i=0}^m$.
As other theorems and remarks in \cite{mesner_association_1990} provide values and restrictions for $p_{ijk}^l$ when $\lr{i,j,k,l} \cap \lr{0,1,2,3} \neq \varnothing$, \cite{mesner_association_1990} suggests that the most interesting intersection numbers are those that arise from the above subalgebra.

One open problem is to determine whether or not the desirable algebraic properties of classical association schemes hold for ASTs, such as the semisimplicity and duality properties of their adjacency algebras. In \cite{lister1971ternary}, Lister develops a structure theory of associative ternary algebras, with analogues of ideals, modules, identities, inverses, and decompositions. In particular, he provides ternary analogues for fields and semisimplicity. The simplest case where we may apply \cite{lister1971ternary} occurs when the given AST has only one nontrivial relation. \begin{ex}
Let $X=\lr{R_i}_{i=0}^4$ be an AST with only one nontrivial relation $R_4$. In this case, the ternary subalgebra generated by $A_4$ is both associative and commutative. In fact, with $A_4 A_4 A_4 = p_{444}^4 A_4$, we see that $\frac{1}{p_{444}}A_4$ and $A_4$ is a ternary identity of the subalgebra provided that $p_{444}^4\neq 0$. Moreover, with $c\neq 0$, $c A_4$ has an inverse $\frac{1}{c p_{444}^4} A_4 $. Therefore, when there is only one nontrivial relation $R_4$, the subalgebra generated by $A_4$ is a ternary field.  
\end{ex} 

Beyond this, we currently have no examples of ASTs whose subalgebras generated by the adjacency hypermatrices of the nontrivial relations are associative. As such, we have yet to acquire nontrivial examples on which to apply the structure theory of \cite{lister1971ternary}. A possible means to circumvent this may be taken from \cite{mesner_ternary_1994}, where a different notion of identities and inverses, called identity pairs and inverse pairs, are defined. A means of computation of such pairs are also given, with applications to ASTs obtained from 2-designs and some two-transitive groups. Although the authors did not provide a direct application of such pairs to ASTs, they found that computing for inverse pairs is equivalent to solving certain equations with the structure constants of the ASTs, and that there are plenty of inverse pairs from the subalgebras of ASTs. More recently, Gnang derived explicit necessary and sufficient conditions for the existence of inverse pairs and used these to formulate and prove a ternary analogue for the usual rank-nullity theorem \cite{GNANG2020391}.

If we instead focus on commutativity, there are nontrivial examples of commutative association schemes that will be given in the later sections. In particular, the ASTs from 2-designs, and the ASTs from the projective linear groups and the sporadic two-transitive groups are commutative. However, the structural properties of the subalgebras afforded by their commutativity remain unknown. In the classical association scheme case, the semisimplicity and duality properties may be obtained from applications of the spectral decomposition theorem on the mutually commuting adjacency matrices \cite{bannai_algebraic_1984}. Although they have yet to be applied to study the structure of the subalgebras of ASTs, there exist analogues of the fundamental theorem of linear algebra, the spectral theorem, and the Gram-Schmidt orthogonalization process for cubic hypermatrices \cite{GNANG2017238,AFST_2011_6_20_4_801_0}.

\section{ASTs from group actions} \label{section_grp}

In this section we consider ASTs whose relations are invariant under the action of some group.

\subsection{ASTs from two-transitive groups}
Analogous to Schurian association schemes, which are classical association schemes obtained from transitive group actions \cite{bannai_algebraic_1984}, an AST may arise from the action of a two-transitive group. Indeed, given a two-transitive group $G$ acting on a set $\Omega$, the orbits of the induced action on $\Omega \times \Omega \times \Omega$ is an AST \cite{mesner_association_1990}. Further, the orbits of a two-point stabilizer are in bijection with the nontrivial relations of the AST and the sizes of these orbits give the third valencies \cite{mesner_association_1990,balmafamily}. 

In \cite{mesner_association_1990}, the sizes of the ASTs from the affine group $AGL(1,n)$, the projective group $PSL(2,n)$, the Suzuki group $Sz(2^{2k+1})$, and the Higman-Sims group $HS$ were acquired. Some intersection numbers from the ASTs obtained from $AGL(1,n)$ and $PSL(2,n)$ were also obtained. These were extended in \cite{balmafamily}, where the sizes and third valencies of the ASTs obtained from $S_n$ and $A_n$, $PGU(3,q)$, $PSU(3,q)$, and $Sp(2k,2)$, $Sz(2^{2k+1})$ and $Ree(3^{2k+1})$, some subgroups of $A\Gamma L(k,n)$, some subgroups of $P\Gamma L(k,n)$, and the sporadic two-transitive groups, were obtained. The intersection numbers of the ASTs from these subgroups of $P\Gamma L(k,n)$ and $A \Gamma L(k,n)$, and the sporadic two-transitive groups were also determined. In particular, the ASTs obtained from the projective linear groups and the sporadic two-transitive groups were found to be commutative. 

As of yet, the intersection numbers of the ASTs from the projective unitary groups, the symplectic groups, the Suzuki groups, and the Ree groups remain undetermined. The sizes, third valencies, and intersection numbers of the two-transitive subgroups of $A\Gamma L(k,n)$ not of the form $AGL(k,n)\rtimes H$ (where $H\leq Gal(GF(n))$) also remain undetermined.

\subsection{Symmetric and Circulant ASTs}
A ternary relation $R\subseteq \Omega\times \Omega \times \Omega$ is called symmetric if it is invariant under coordinate permutation; that is, for any $\sigma\in S_3$, we have \[R=\lr{(x_{\sigma(1)},x_{\sigma(2)},x_{(\sigma(3))}):(x,y,z)\in R}.\] An AST is called symmetric if all its nontrivial relations are symmetric \cite{mesner_association_1990}. In \cite{mesner_association_1990}, some parameters of such ASTs are computed. Through these computations, the authors found that a weak associative law holds for symmetric ASTs: \[(A_i A_i A_j) A_i A_i = A_i (A_i A_j A_i) A_i  =A_i A_i (A_j A_i A_i).\] Some examples of symmetric ASTs are the ASTs with only one nontrivial relation, and ASTs obtained from 2-designs and two-graphs \cite{mesner_association_1990}, as will be discussed in the next section.

Motivated by symmetric ASTs, another family of ASTs called circulant ASTs was defined in \cite{Zealand2021}. The nontrivial relations of these ASTs are called circulant, being invariant under the action of a common transitive cyclic subgroup of $S_n$. It turns out that such ASTs correspond to partitions of a certain subset of $\Omega^{[2]}=\lr{(x,y):x\neq y}\subseteq\Omega \times \Omega$ called AST-regular partitions. In particular, enumerating the AST-regular partitions yields all circulant ASTs over $\Omega$. This suggests looking for AST-regular partitions from known families of circulant ASTs, particularly the ASTs obtained from $PSL(2,q)$ for $q$ even and $AGL(1,q)$ for $q$ prime. 

The authors of \cite{Zealand2021} also defined the notion of thinness for ternary relations, wherein a circulant ternary relation $R$ is said to be $ab$-thin provided that the mapping $\sigma_{ab}$ from $R$ to $\Omega\times \Omega$ given by $\sigma_{ab}(x_1,x_2,x_3)=(x_a,x_b)$ is injective with image $\Omega^{[2]}$. By reasoning with thin relations, they showed that any nontrivial relation of a circulant AST is a disjoint union of thin circulant ternary relations. In light of this, \cite{Zealand2021} suggests finding circulant ASTs where every nontrivial circulant relation is thin.

The authors of \cite{Zealand2021} also suggested finding further examples of circulant ASTs. We give a new example below, obtained through GAP 4.11.1, which is the AST obtained from the action of the permutation representation of $PSL(2,11)$ of degree 11. It is a commutative AST whose parameters are given in \cite{balmafamily}.

\begin{ex}
Let $G$ be the subgroup of $S_{11}$ produced by the following generators: \[(1,2,4,9,5,7,3,11,10,6,8),\;(1,2,5,8,9)(4,10,7,6,11), \text{ and } (1,2,5,8,10,11)(3,6,7)(4,9).\] Then $G$ is permutation isomorphic to the degree 11 representation of $PSL(2,11)$ as described in \cite{dixon1996permutation}. Further, the AST obtained from $G$ is circulant, as each of its nontrivial relations is invariant under the action of the subgroup of $S_{11}$ generated by the length 11 cycle $( 1, 2, 3,10, 6,11, 8, 4, 5, 9, 7)$. 
\end{ex}

\subsection{ASTs with relations invariant under some group}
Motivated by the classification of classical association schemes over small vertices and the computational simplifications that arise by considering ASTs whose relations are invariant under some group action, \cite{balmasmall} provided an algorithm for generating (up to isomorphism) all ASTs over a given number of vertices whose nontrivial relations are invariant under some predetermined group action. In particular, appropriate choices of group actions can enumerate the symmetric ASTs, circulant ASTs, and even all ASTs over a given number of vertices, provided sufficient computational resources are available.


The authors found that there is a unique AST over three vertices, a unique symmetric ASTs over four or five vertices, a unique AST over four vertices with two nontrivial relations, and a unique nontrivial circulant AST over five vertices. Since the current version of the algorithm is computationally expensive, it would benefit from improvements either in the algorithm itself or in its implementation.

\section{ASTs from 2-designs and two-graphs}\label{section_combi}
Although there are no constructions at this time of ASTs that serve as analogues of the Hamming scheme or the Johnson scheme from classical association schemes, there are relationships between ASTs and other combinatorial objects such as two-graphs and 2-designs \cite{mesner_association_1990}.
\subsection{2-designs} 
A 2-design $(\Omega,B)$ with parameters $b,v,k,\lambda$ is a family $B$ of $k$-subsets of a set $\Omega$ (called blocks) such that any $2$-subset of $\Omega$ lies in exactly $\lambda$ blocks, and where $|\Omega|=v$, and $|B|=b$. In \cite{mesner_association_1990}, it was found that any symmetric nontrivial relation $R_i$ of an AST yields a family of 2-designs. Indeed, the family $B$ of all 3-subsets $\lr{x,y,z}$ such that $(x,y,z)\in R_i$ is a 2-design. Moreover, a 2-design $B$ with $\lambda=1$ yields an AST, by letting the nontrivial relations be \[R_4 = \lr{(x,y,z):x,y,z\;\text{distinct, }\lr{x,y,z}\;\text{lies in some block of }B}\] and \[R_5=\lr{(x,y,z):x,y,z\;\text{distinct, }\lr{x,y,z}\;\text{does not lie in any block of }B}.\]
The authors then proved that this subalgebra generated by the adjacency matrices $A_4$ and $A_5$ is commutative but not associative. It is also not minimal, for $A_4$ generates a proper subalgebra not containing $A_5$. As a converse, an AST with two nontrivial relations that are both symmetric and whose parameters satisfy certain conditions can be recovered from a 2-design through the above construction \cite{mesner_association_1990}. 

The construction above is for designs with $\lambda=1$. However, it is generally not possible to proceed in a similar manner for larger values of $\lambda$, although the construction holds for certain choices of $\lambda$-designs \cite{mesner_association_1990}. It remains undetermined which $\lambda$ and $\lambda$-designs permit such a construction of an AST.

\subsection{Two-graphs}
A collection $\Delta$ of $3$-subsets of a set $\Omega$ such that every $4$-subset of $\Omega$ contains an even number of members of $\Delta$ is called a two-graph $(\Omega,\Delta)$. A two-graph is regular if each pair of elements of $\Omega$ is contained in the same number of triples of $\Delta$. In \cite{mesner_association_1990}, the authors showed that an AST with two nontrivial relations $R_4$ and $R_5$ that are both symmetric and whose parameters satisfy \[p_{445}^4=p_{454}^4=p_{544}^4=p_{555}^4=p_{444}^5=p_{455}^5=p_{545}^5=p_{554}^4=0\] is equivalent to a regular two-graph.

\section{Fusion and fission ASTs}\label{section_sub}
Motivated by classical fusion and fission association schemes \cite{BANNAI1993385}, we define fusion and fission ASTs, providing a relationship between fusion and fission ASTs and a number of examples. In particular, we give the AST obtained from $ASL(2,q)$ as an instance of a fission scheme of the AST obtained from the affine general linear group $AGL(2,q)$. Further, we obtain the parameters of the AST from $ASL(2,q)$, extending the results of \cite{balmafamily} to another infinite family of two-transitive groups. We begin with the definition of fusion and fission ASTs, ASTs that can be obtained from combining or splitting the relations of another AST.

\begin{defi}
Let $\Omega$ be a nonempty set and let $X=\lr{R_i}_{i=0}^m$ and $\Tilde{X}=\lr{\Tilde{R}_\alpha}_{\alpha=0}^n$ be ASTs on $\Omega$. If for each $ i\in \lr{0,\ldots, m}$ there exists an $ \alpha \in \lr{0,\ldots,n}$ such that $R_i \subseteq \Tilde{R}_\alpha$, then we say that $\Tilde{X}$ is a fusion AST of $X$ and $X$ is a fission AST of $\Tilde{X}$.
\end{defi}

In other words, the relations of a fusion AST $\Tilde{X}$ are unions of the relations of its fission AST $X$. The following theorem provides a relationship between the intersection numbers of a fission AST and a fusion AST. 

\begin{thm}\label{thm_fus}
Let $\Omega$ be a nonempty set, $X=\lr{R_i}_{i=0}^m$ be an AST on $\Omega$, and $\Tilde{X}=\lr{\Tilde{R}_\alpha}_{\alpha=0}^n$ be a fusion AST of $X$. For each $\alpha\in \lr{0,\ldots,n}$, let $\Lambda_\alpha = \lr{i \in \lr{0,\ldots,m}: R_i \subseteq \Tilde{R}_\alpha}$. If $\Tilde{p}_{\alpha\beta\gamma}^{\delta}$ is the intersection number corresponding to $\Tilde{R}_\alpha$, $\Tilde{R}_\beta$, $\Tilde{R}_\gamma$, and $\Tilde{R}_\delta$, then \[\Tilde{p}_{\alpha\beta\gamma}^{\delta}=\sum_{i \in \Lambda_\alpha}\sum_{j \in \Lambda_\beta}\sum_{k \in \Lambda_\gamma} p_{ijk}^l\] for any $l \in \Lambda_\delta$. Furthermore, the third valency of a nontrivial relation $\Tilde{R}_\epsilon$ of $\Tilde{X}$ is $\Tilde{n}_\epsilon^{(3)}=\sum_{i \in \Lambda_\epsilon} n_i^{(3)}$.
\end{thm}

\begin{proof}
Fix any $(x,y,z)\in R_l \subseteq \Tilde{R}_\delta$. The intersection number $\Tilde{p}_{\alpha\beta\gamma}^{\delta}$ is equal to the number of $w$ such that $(w,y,z)\in \Tilde{R}_\alpha$, $(x,w,z)\in \Tilde{R}_\beta$, and $(x,y,w)\in \Tilde{R}_\gamma$. However, these conditions on $w$ are equivalent to $(w,y,z)\in R_i$ for some $i\in \Lambda_\alpha $, $(x,w,z)\in R_j$ for some $j\in \Lambda_\beta$, and $(x,y,w)\in R_k$ for some $k\in \Lambda_\gamma $. This yields \[\Tilde{p}_{\alpha\beta\gamma}^{\delta}=\sum_{i \in \Lambda_\alpha}\sum_{j \in \Lambda_\beta}\sum_{k \in \Lambda_\gamma} p_{ijk}^l.\] Finally, given a nontrivial relation $\Tilde{R}_\epsilon$ of $\Tilde{X}$ and a member $(x,y,z)$ of $\Tilde{R}_\epsilon$, the third valency $\Tilde{n}_\epsilon^{(3)}$ is equal to the number of $w$ such that $(x,y,w)\in \tilde{R}_\epsilon$. This condition on $w$ is equivalent to $(x,y,w)$ being in $R_i$ for some $i\in \Lambda_\epsilon$, yielding $\Tilde{n}_\epsilon^{(3)}=\sum_{i \in \Lambda_\epsilon} n_i^{(3)}$.
\end{proof}

Theorem \ref{thm_fus} can be expressed in terms of the adjacency hypermatrices of $X$ and $\Tilde{X}$. If for each $\alpha\in \lr{0,\ldots,n}$ we let $\Tilde{A}_\alpha$ denote the adjacency hypermatrix corresponding to $\Tilde{R}_\alpha \in \Tilde{X}$, then $\Tilde{A}_\alpha = \sum_{i\in \Lambda_\alpha} {A_i} $. Hence, we have the following equality for any $\alpha,\beta,\gamma \in \lr{0,\ldots,n}$. 
\[\Tilde{A}_\alpha \Tilde{A}_\beta \Tilde{A}_\gamma 
= \sum_{\delta=0}^n \Tilde{p}_{\alpha\beta\gamma}^\delta \Tilde{A}_\delta 
= \sum_{\delta=0}^n \sum_{l\in \Lambda_\delta} \Tilde{p}_{\alpha\beta\gamma}^\delta A_l
=\sum_{\delta=0}^n \sum_{l\in \Lambda_\delta}
\left( \sum_{i \in \Lambda_\alpha}\sum_{j \in \Lambda_\beta}\sum_{k \in \Lambda_\gamma} p_{ijk}^l \right)A_l
         .\]

We now give some examples of fusion and fission ASTs. First, the AST over a set $\Omega$ with only one nontrivial relation is a fusion scheme of any AST over $\Omega$.

\begin{ex}
Let $X=\lr{R_i}_{i=0}^m$ be any AST on a set $\Omega$. Let $\Tilde{X}=\lr{R_0,R_1,R_2,R_3,\cup_{i=4}^m R_i}$ be the given partition of $\Omega \times \Omega \times \Omega$. Then $\Tilde{X}$ is a fusion AST of $X$ equal to the only AST over $\Omega$ with only one nontrivial relation.
\end{ex}

\begin{ex} 
The following example from two-graphs is obtained from Remark 5.6 of \cite{mesner_association_1990}. If an AST $X=\lr{R_\alpha}_{\alpha=0}^n$ that has more than two nontrivial relations contains a subset $J\subseteq \lr{4,5,\ldots,m}$ of indices such that ${R}_i$ for each $i \in J$ is symmetric, define \[\Delta = \lr{\lr{a,b,c}:(a,b,c)\in R_i \text{ for some } i\in J}.\] If for any quadruple $i,j,k,l$ from $\Omega$ such that an odd number of these lie in $J$ we have ${p}_{ijk}^l=0$, then $\Delta$ is a regular two-graph. This two-graph then produces an AST $\Tilde{X}$ with two-nontrivial relations, each of which is a union of the nontrivial relations of $X$; that is, $\Tilde{X}$ is a fusion AST of $X$.  
\end{ex}

The next example shows that fission ASTs occur naturally from two-transitive subgroups of two-transitive groups.
\begin{ex}
Let $G$ be a two-transitive group acting on $\Omega$ and $H$ be a two-transitive subgroup of $G$. Let $\Tilde{X}$ be the AST obtained from the action of $G$ on $\Omega$ and $X$ be the AST obtained from the action of $H$ on $\Omega$. Then $X$ is a fission scheme of $\Tilde{X}$. Indeed, the relations of $X$ are the orbits of $H$ on $\Omega \times \Omega \times \Omega$ while the relations of $\Tilde{X}$ are the orbits of $G$ on $\Omega \times \Omega \times \Omega$. Since $H\leq G$, the orbits of $G$ on $\Omega \times \Omega \times \Omega$ are unions of orbits of $H$. 
Examples of such fission ASTs occuring from two-transitive subgroups of two-transitive groups, along with some of their parameters, are in \cite{balmafamily}.
\end{ex}

For the final illustration, we compute the parameters of the AST obtained from the affine special linear group $ASL(2,q)$ where $q$ is a prime power, extending the work done in \cite{balmafamily}. Since $ASL(2,q)$ is a two-transitive subgroup of $AGL(2,q)$, the following is also an example of a fission scheme and its parameters. The proofs are omitted, the methods being similar to the ones used in \cite{balmafamily}.

For ease of discussion, we fix the following notations. Let $q$ be a prime power and $X$ be the AST obtained from \[ASL(2,q)= \lr{(x,y)^T \mapsto A(x,y)^T: A\in SL(2,q)}, \] the group of affine transformations on the affine space $V=(GF(q))^2$ whose linear parts have determinant 1. For $a\in GF(q)$, let $\vec{a}=(a,0)^T\in V$. Additionally, for $(u,v,w)\in V\times V\times V$, let $[(u,v,w)]\in X$ denote the orbit of $(u,v,w)$ under $ASL(2,q)$. The size and third valencies of $X$ are given in the following theorem. 

\begin{thm}
Let $q$ be a prime power and $X$ be the AST obtained from the action of $ASL(2,q)$. The two-point stabilizer $ASL(2,q)_{\vec{0},\vec{1}}$ has $2q-3$ orbits on $V\setminus \lr{\vec{0},\vec{1}}$. There are $q-2$ orbits of the form $\lr{\vec{a}}$, where $a\neq0,1$, and $q-1$ orbits of the form $\lr{(c,\mathfrak{a})^T:c\in GF(q)}$, where $\mathfrak{a}\neq 0$. Thus, $X$ has $2q-3$ nontrivial relations. There are $q-2$ nontrivial relations of the form \[R^a = \lr{[(\vec{0},\vec{1},\vec{a})]}, \; a\neq0,1,\] each with third valency 1. The remaining $q-1$ nontrivial relations of $X$ are of the form  \[{}^\mathfrak{a}\!R=\lr{[(\vec{0},\vec{1},(0,\mathfrak{a})^T)]}, \; \mathfrak{a}\neq0,\] each with third valency $q$. 
\end{thm}

For notational convenience, let $A^a$ denote the adjacency hypermatrix corresponding to the relation $R^a$ whenever $a\neq0,1$. Similarly, let ${}^\mathfrak{a} \! A$ denote the adjacency hypermatrix corresponding to the relation ${}^\mathfrak{a} \!R$ whenever $\mathfrak{a} \neq 0$. The intersection numbers of the subalgebra generated by the adjacency hypermatrices of the nontrivial relations of $X$ are given implicitly in the next theorem. 

\begin{thm}
Let $q$ be a prime power and $X$ be the AST obtained from the action of $ASL(2,q)$. The following equations hold for any $a,b,c\neq0,1$ and $\mathfrak{a},\mathfrak{b},\mathfrak{c}\neq 0$. 
\begin{enumerate}
    \item $A^a A^b A^c, = \begin{cases}A^{bc} & \text{if }bc=a(1-c)+c\neq1, \\ 0, & \text{otherwise}. \end{cases}$ 
    \item $A^a A^b \,{}^\mathfrak{c}\!A= A^a\, {}^\mathfrak{c}\!A\, A^b={}^\mathfrak{c}\!A \, A^a  A^b=0  $.
    \item ${}^\mathfrak{a}\!A \, {}^\mathfrak{b}\! A \,A^c = \begin{cases}{}^{\frac{\mathfrak{b}}{c}}\!A & \text{if }\mathfrak{a}c+\mathfrak{b}c=\mathfrak{b}, \\ 0, & \text{otherwise}. \end{cases}$  
    \item ${}^\mathfrak{a}\!A \, A^c\, {}^\mathfrak{b}\! A  = \begin{cases}{}^{{\mathfrak{b}}{c}}\!A & \text{if }\mathfrak{b}c=\mathfrak{a}+\mathfrak{b}, \\ 0, & \text{otherwise}. \end{cases}$  
    \item $A^c\,{}^\mathfrak{a}\!A \, {}^\mathfrak{b}\! A  = \begin{cases}{}^{{\mathfrak{b}}{(1-c)}}\!A & \text{if }\mathfrak{a}=-\mathfrak{b}c, \\ 0, & \text{otherwise}. \end{cases}$
    \item ${}^\mathfrak{a}\!A \, {}^\mathfrak{b}\! A \, {}^\mathfrak{c}\!A = \begin{cases}q A^{-\frac{\mathfrak{b}}{\mathfrak{c}}} & \text{if }\mathfrak{a}+\mathfrak{b}+\mathfrak{c}=0, \\ {}^{\mathfrak{a}+\mathfrak{b}+\mathfrak{c}}\!A & \text{if }\mathfrak{a}+\mathfrak{b}+\mathfrak{c}\neq0. \end{cases}$
\end{enumerate}

\end{thm}

The last theorem gives the intersection numbers $p_{ijk}^l$ of the ASTs obtained from $ASL(2,q)$ whenever exactly one of $R_i$, $R_j$, and $R_k$ is trivial. Here $I_1$, $I_2$, and $I_3$ denote the respective adjacency hypermatrices of the trivial relations $R_1$, $R_2$, and $R_3$ of $X$.

\begin{thm}
Let $q$ be a prime power and $X$ be the AST obtained from the action of $ASL(2,q)$. The following equations hold for any $a,b\neq0,1$ and $\mathfrak{a},\mathfrak{b}\neq 0$. 
\begin{enumerate}
    \item $I_1 A^a A^b = \begin{cases} I_1, &\text{if }ab=1, \\ 0, &\text{otherwise.}\end{cases}$
    \item $A^a I_2 A^b = \begin{cases} I_2, &\text{if }ab=a+b, \\ 0, &\text{otherwise.}\end{cases}$
    \item $A^a A^b I_3 = \begin{cases} I_3, &\text{if }a+b=1,  \\ 0, &\text{otherwise.}\end{cases}$
    \item $I_1 A^a\, {}^{\mathfrak{a}}\!A = I_1 \, {}^{\mathfrak{a}}\!A\, A^a  = A^a I_2 \, {}^{\mathfrak{a}}\!A = \, {}^{\mathfrak{a}}\!A\, I_2 A^a =A^a \, {}^{\mathfrak{a}}\!A\, I_3 = \, {}^{\mathfrak{a}}\!A\, A^a I_3 =0$.
    \item $I_1 \, {}^{\mathfrak{a}}\!A\, {}^{\mathfrak{b}}\!A = \begin{cases}qI_1 & \text{if }\mathfrak{a}=\mathfrak{-b}, \\ {0}, & \text{otherwise}. \end{cases}$

    \item $ {}^{\mathfrak{a}}\!A\, I_2 \, {}^{\mathfrak{b}}\!A = \begin{cases}qI_2 & \text{if }\mathfrak{a}=\mathfrak{-b}, \\ {0}, & \text{otherwise}. \end{cases}$
    \item $ {}^{\mathfrak{a}}\!A\, {}^{\mathfrak{b}}\!A\,I_3  = \begin{cases}qI_3 & \text{if }\mathfrak{a}=\mathfrak{-b}, \\ {0}, & \text{otherwise}. \end{cases}$
\end{enumerate}

\end{thm}


\bibliographystyle{amsplain}
 \bibliography{ternassoc}





\end{document}